\documentclass
[12pt]{article} \oddsidemargin 0in \evensidemargin 0in \textheight
9in \textwidth 6.5in \topmargin 0in \headheight 0in
\usepackage{amsmath,amsfonts,amsthm,amssymb}

\newtheorem{thm}{Theorem}
\newtheorem{theorem}{Theorem}[section]
\newtheorem{cor}[theorem]{Corollary}

\newtheorem{lemma}[theorem]{Lemma}

\def\irr#1{{\rm Irr}(#1)}
\def\irrr#1#2 {\irr {#1 \mid #2}}

\title{A central series associated with $V(G)$}

\author {Nabil M.\ Mlaiki
   \\ {\it Department of Mathematical Sciences, Kent State University}
   \\ {\it Kent, Ohio 44242}
   \\ E-mail: nmlaiki@math.kent.edu
      }

\date{}

\begin{document}

\maketitle{}

\begin{abstract}
We generalize  Lewis's result about a central series associated with
the vanishing off subgroup. We write $V_{1}=V(G)$ for the vanishing
off subgroup of $G$,  and $V_{i}=[V_{i-1},G]$ for the terms in this
central series. Lewis proved that there exists a positive integer
$n$ such that if $V_{3} < G_{3}$, then
$|G:V_{1}|=|G':V_{2}|^{2}=p^{2n}$. Let $D_{3}/V_{3} =
C_{G/V_{3}}(G'/V_{3})$. He also showed  that if $V_{3} < G_{3}$,
then either $|G:D_{3}|=p^{n}$ or $D_{3}=V_{1}$. We show that if
$V_{i} <G_{i}$ for $i\ge 4,$ where $G_{i}$ is the $i$-th term in the
lower central series of $G$, then $|G_{i-1}:V_{i-1}|=|G:D_{3}|.$

\end{abstract}

\section{Introduction}

Throughout this paper, $G$ is a finite group. We write
${\rm{Irr}}(G)$ for the set of irreducible characters of $G$ and
${\rm{nl}}(G)=\{\chi\in {\rm{Irr}}(G) \mid \chi(1)\neq 1\}.$ Define
the vanishing off subgroup of $G$, denoted by $V(G) ,$ by
$V(G)=\langle g\in G \mid$ there exists $\chi \in$ nl$(G) $ such
that $ \chi(g) \neq 0 \rangle.$ This subgroup was first introduced
by Lewis in \cite{Lewis1}. Note that $V(G)$ is the smallest subgroup
of $G$ such that all nonlinear irreducible characters vanish on
$G\setminus V(G).$ Consider the term $G_{i}$ as the $i$-th term in
the lower central series, which is defined by $G_{1}=G,$
$G_{2}=G'=[G,G],$ and $G_{i}=[G_{i-1},G]$ for $i\ge 3.$ We are going
to study a central series associated with the vanishing off
subgroup, defined inductively by $V_{1}=V(G)$, and $V_{i}=[ V_{i-1},
G ]$ for $i\ge 2.$ Lewis proved in \cite{Lewis1} that $G_{i+1}\le
V_{i}\le G_{i}.$ In \cite{Lewis1}, Lewis showed that when
$V_{i}<G_{i},$ we have $V_{j}<G_{j}$ for all $j$ such that $1\le
j\le i.$ Also, in \cite{Lewis1}, Lewis proved that if $V_{2} <
G_{2}$, then there exists a prime $p$ such that $G_{i}/V_{i }$ is an
elementary abelian $p$-group for all $i\ge 1.$ In addition, he
proved that there exists a positive integer $n$ such that if $V_{3}
< G_{3}$, then $|G:V_{1}|=|G':V_{2}|^{2}=p^{2n}$. We are able to
generalize the results in \cite{Lewis1} to the case where
$V_{i}<G_{i}$ for $i>3.$ Also, we prove that the index of $V_{i-1}$
in $G_{i-1}$ is the same as the index of $D_{3}$ in $G.$

We define some subgroups that are useful to prove our results.
First, set $D_{3}/V_{3} = C_{G/V_{3}}(G'/V_{3}).$ Lewis proved in
\cite{Lewis1} that if $V_{3} < G_{3}$, then either
$|G:D_{3}|=\sqrt{|G:V_{1}|}$ or $D_{3}=V_{1}$. To study the case
when $i>3,$ we define some more subgroups. For each integer $i\ge3,$
set $ Y_{i}/V_{i}=Z(G/V_{i})$ and
$D_{i}/V_{i}=C_{G/V_{i}}(G_{i-1}/V_{i}).$

We say $G_{k}$ is $H_{1}$, if for every normal subgroup $N$ of $G$
where $V_{k}\le N<G_{k}$ we have $V_{k-1}/N=G_{k-1}/N\cap
Y_{k}(G/N).$ In \cite{Lewis1}, it was proved that $G_{3}$ is
$H_{1}.$ Under the additional hypothesis that $G'/V_{i}$ is abelian,
we are able to show that $G_{i}$ is $H_{1}$ for all $i>3.$ We are
also interested in computing the index of $V_{i}$ in $G_{i}$. We
will see that this index depends on the size of $D_{3}.$ In other
words, it depends on the size of the centralizer of $G'$ modulo
$V_{i}.$

We now come to our first theorem. When $V_{k}<G_{k}$, and $G_{i}$ is
$H_{1}$ for $i= 4,\cdots,k,$ we are able to prove that
$D_{k}=D_{3},$ which is very useful to prove some of the results of
this paper.
\begin{thm}
Assume that $V_{k}< G_{k}$, $G'/V_{k}$ is abelian, and $G_{i}$ is
$H_{1}$ for all $i=3,\cdots,k$. Then $D_{k}=D_{3}.$
\end{thm}

Our second theorem should be considered to be the main result of
this paper. We are able to prove that $|G_{i-1}:V_{i-1}|=|G:D_{3}|,$
for every $i \ge 4$, where $V_{i}< G_{i}$, and $G'/V_{i} $ is
abelian. Hence, for a nilpotent group of class $c,$ if $V_{c}
<G_{c},$ and $G'/V_{c}$ is abelian, then we have
$|G_{i-1}:V_{i-1}|=|G:D_{3}|$ for all $4\le i\le c,$ and
$|G_{c}:V_{c}|\le |G:D_{3}|.$
\begin{thm}
Assume that $V_{k}< G_{k}$, $G'/V_{k} $ is abelian, for some $k\geq3.$ Then\\
(a) $|G_{k-1}:V_{k-1}|=|G:D_{3}|$ for $k\ge 4.$\\
(b) $D_{k}=D_{3}.$\\
(c) $G_{k}$ is $H_{1}$.\\
(d) $|G_{k}:V_{k}|\le |G:D_{3}|.$
\end{thm}

Let $G$ be a finite group, we say that $G$ is a Camina group if
$cl(x)=xG'$ for every $x\in G\setminus G'.$  If $3\le i\le k-1,$
then $V_{i}$ will satisfy the same hypothesis as a Camina group. So,
$D_{i}=D_{3},$ $G_{i}$ is $H_{1}$ and when $i\ge 4,$
$|G_{i-1}:V_{i-1}|=|G:D_{3}|.$ Note that the above result was
motivated from the bound of subgroups by MacDonald in
\cite{MacDonald1}, where he proved that $|G_{3}|\le |G:G'|$ for a
Camina group $G.$  Our motivation for adding the hypothesis
$G/V_{k}$ abelian is that the results in \cite{MacDonald1} were
under the hypothesis that $G$ is metabelian (i.e., $G'$ is abelian.)
Hence, proving this conclusion under a similar metabelian hypothesis
seems like a reasonable first step. In the Camina group case,
removing the metabelian hypothesis required totally different
techniques.

In closing, as an application of our techniques we answer an open
question about Camina groups. In \cite{MacDonald1}, MacDonald
conjectured that if $G$ is a Camina group of nilpotence class $3,$
then $|G_{3}| \le p^{n},$ where $p^{2n}=|G:G'|.$ He gave a sketch of
a proof. But Dark and Scoppola observed in \cite{Drak1} that
MacDonald's proof was not conclusive. So, they proved that if $G$ is
a Camina group of nilpotence class $3,$ then $|G_{3}| \le
p^{\frac{3n}{2}}.$ In our third theorem, we give a conclusive proof
of MacDonald's conjecture.
\begin{thm}
If $G$ is a Camina group of nilpotence class $3$ with
$|G:G'|=p^{2n},$ then $|G_{3}| \le p^{n}.$
\end{thm}

Acknowledgement: I would like thank my advisor, Dr. Mark Lewis, 
for his input and the useful weekly discussions regarding this
paper. This research is a part of my doctoral dissertation.

\section{General Lemmas}
In this section, we prove some lemmas that are useful for the proofs
of our theorems. Also, some of these facts give us a good idea about
the relation between the lower central series and the central series
associated with the vanishing off subgroup that we defined in the
introduction. Lewis showed in \cite{Lewis1} that both series are
related by proving that $V_{i}\le G_{i}\le V_{i-1}.$ We now show
that if $G_{k}$ is $H_{1},$ then $V_{k-1}=G_{k-1}\cap Y_{k}.$
\begin{lemma}\label{Hone}
Assume that $V_{k}<G_{k}.$ If there exists $N$ such that $ V_{k}\le
N<G_{k}$ with $V_{k-1}/N=(G_{k-1}/N) \cap Z(G/N),$ then
$V_{k-1}=G_{k-1}\cap Y_{k}.$
\end{lemma}
\begin{proof}
Observe that $Y_{k}/N\le Z(G/N).$ We have
$$V_{k-1}/N \le (Y_{k}\cap G_{k-1})/N = (Y_{k}/N) \cap (G_{k-1}/N) \le
Z(G/N)\cap (G_{k-1}/N) =V_{k-1}/N.$$ Thus, we obtain equality
throughout, and $V_{k-1}=G_{k-1}\cap Y_{k}$ as desired.
\end{proof}
As an immediate consequence, note that if $G_{k}$ is $H_{1},$ then
$V_{k-1}=G_{k-1}\cap Y_{k}.$ This next lemma is well known.
\begin{lemma}\label{twoone}
If $G$ is nilpotent and $|G_{i}|=p,$ then for every $x \in
G_{i-1}\setminus( G_{i-1}\cap Y_{i})$, we have $cl(x)=xG_{i}.$
\end{lemma}
\begin{proof}
Because $G$ is nilpotent, we can write $G=P\times Q$ where $P$ is a
$p$-group and $Q$ is a $p'$-group. Hence, $G_{i-1}= P_{i-1}\times
Q_{i-1}.$ As $|G_{i}|=p$, we have $G_{i}=P_{i}$. In particular,
$Q_{i-1}\le Z(G).$ Observe that $G_{i-1}/G_{i}$ is central in
$G/G_{i}.$ Thus, it follows that $cl(x) \subseteq xG_{i}.$ We deduce
that $|cl(x)|\leq p$. Recall that $x \in G_{i-1}\setminus Y_{i},$
which implies that $Q \le C_{G}(x)$. Now, $|cl(x)|=|G:C_{G}(x)|$
divides $|G:Q|=|P|$. Therefore, $|cl(x)|$ is either $1$ or $p$.
Since $x$ is not central, we must have $|cl(x)|=p=|xG_{i}|.$ We
conclude that $cl(x)=xG_{i}.$
\end{proof}

Now, we get a relationship between the central series associated
with the vanishing off subgroup of the whole group and a quotient
group of that group.
\begin{lemma}\label{twothree}
Assume that $V_{k}<G_{k},$ for some $k\ge 3.$ Then for every normal
subgroup $ N<G_{k}$ we have $V_{i}(G/N)=V_{i}/N$ for every $2\leq i
\leq k.$
\end{lemma}
\begin{proof}
We prove this by induction. In Lemma 2.2 in \cite{Lewis1}, we have
$V_{1}(G/V_{2})= V(G)/V_{2}.$ Let $X/N=V(G/N).$ By Lemma 3.3 in
\cite{Lewis1}, $X\le V(G).$ On the other hand, $V_{2}/N$ is normal
in $G/N.$ By Lemma 3.3 in \cite{Lewis1} applied to $G/N$, we have
$V(G)/V_{2} =V_{1}(G/V_{2})= V_{1}((G/N)/(V_{2}/N))\le
V(G/N)/(V_{2}/N)=(X/N)/(V_{2}/N)\cong X/V_{2}.$ So, $V(G)\le X.$ We
deduce that $X=V(G),$ and $V_{2}(G/N)=V_{2}/N.$ This is the initial
case of the induction. Now, suppose that $i>2$ and assume that
$V_{i-1}(G/N)=V_{i-1}/N.$ Therefore, $V_{i}(G/N)=[ V_{i-1}(G/N),
G/N] =[V_{i-1}/N,G/N]=[V_{i-1},G]N/N=V_{i}/N$ as desired.
\end{proof}
Now, we see the importance of the $H_{1}$ hypothesis.
\begin{lemma}\label{twotwo}
If $V_{i} =1$ and $G_{i}$ is $H_{1}$, then for every $x \in G_{i-1}
\setminus V_{i-1}$ we have $cl(x)=xG_{i}.$
\end{lemma}
\begin{proof}
Since $V_{i}=1$, we have $G_{i}$ is central in $G.$ Thus, $[x,G]$ is
central. This implies that $[x,G]=\{x^{-1}x^{g}\mid g\in G\}.$ It
follows that the map $a\mapsto x^{-1}a$ is a bijection from $cl(x)$
to $[x,G].$ Hence, $cl(x)=xG_{i}$ if and only if $[x,G] = G_{i}$.
Since $x\in G_{i-1},$ it follows that $[x,G] \leq G_{i}.$ Suppose
that $[x,G]<G_{i},$ and we want to find a contradiction. We can find
$N$ such that $[x,G] \leq N < G_{i},$ where $|G_{i}: N|=p$. Since
$x\not\in Y_{i},$ $[x,G]\neq 1.$ Thus, $N>1.$ Applying Lemma
\ref{twothree}, it is not difficult to see that $V_{i-1}(G/N) =
V_{i-1}/N$. Notice that $xN \in Y_{i}(G/N)$. On the other hand, we
have $xN \in G_{i-1}/N = (G/N)_{i-1}.$ Thus, since $G_{i}$ is
$H_{1},$ we have $xN \in Y_{i}(G/N) \cap (G_{i-1}/N) =V_{i-1}(G/N)
\leq V_{i-1}/N$. Therefore, $x \in V_{i-1},$ which contradicts the
choice of $x$.
\end{proof}

The following result is a nice consequence of Lemma \ref{twotwo}
that gives us a good idea about the irreducible characters in
${\rm{Irr}} (G | G_{k}).$

\begin{lemma}
If $V_{k} = 1$ and $G_{k}$ is $H_{1}$, then all the characters in
${\rm{Irr}} (G | G_{k})$ vanish on $G_{k-1} \setminus V_{k-1}$.

\end{lemma}

\begin{proof}
Consider $x\in G_{k-1}\setminus V_{k-1}$.  By Lemma \ref{twotwo} we
have $ cl(x) = xG_{k}$. Applying the second orthogonality relation,
which is Theorem 2.18 in \cite{Isaacs1}, we obtain
$$|G|/|G_{k}|= |G|/|cl(x)|=|C_{G}(x)|= \sum_{\chi \in
{\rm{Irr}}(G)}|\chi(x)|^{2} =  \sum_{\chi \in
{\rm{Irr}}(G/G_{k})}|\chi(x)|^{2} + \sum_{\chi \in {\rm{Irr}}(G\mid
G_{k})}|\chi(x)|^{2}.$$ Since $G_{k-1}/G_{k}$ is central in
$G/G_{k},$ we can use the second orthogonality relation in $G/N$ to
see that
$$ |G:G_{k}| = \sum_{\chi \in {\rm{Irr}}(G/G_{k})}|\chi(xG_{k})|^{2} = \sum_{\chi \in {\rm{Irr}}(G/G_{k})}|\chi(x)|^{2}.$$
Hence, $$\sum_{\chi \in {\rm{Irr}}(G\mid G_{k})}|\chi(x)|^{2} =0.$$
Since $|\chi(x)|^{2} \geq 0$ for each $ \chi \in {\rm{Irr}}(G\mid
G_{k}),$ this implies that all characters in $ {\rm{Irr}}(G\mid
G_{k})$ vanish on $G_{k-1}\setminus V_{k-1}$ as desired.
\end{proof}

Define $E_{i}/(G_{i-1}\cap Y_{i})=C_{G/{(G_{i-1}\cap
Y_{i})}}(G_{i-2}/(G_{i-1}\cap Y_{i})).$ We know that $V_{i-1}\le
G_{i-1}.$ Since $V_{i}=[V_{i-1},G],$ we have $V_{i-1}\le Y_{i},$ and
hence, $V_{i-1}\le G_{i-1} \cap Y_{i}.$ Because
$[G_{i-1},D_{i-1}]\le V_{i-1}\le G_{i-1} \cap Y_{i},$ it follows
that $D_{i-1} \le E_{i}.$

Recall, as a consequence of Lemma \ref{Hone}, that if $G_{i}$ is
$H_{1},$ then $V_{i-1}=G_{i-1}\cap Y_{i}.$ Hence,
$D_{i-1}/V_{i-1}=C_{G/V_{i-1}}(G_{i-2}/V_{i-1})=C_{G/(G_{i-1}\cap
Y_{i})}(G_{i-2}/(G_{i-1}\cap Y_{i}))=E_{i}/(G_{i-1}\cap Y_{i}).$ In
particular, $D_{i-1}=E_{i}.$

Notice that our next lemma is the only time we use the hypothesis
$G'/V_{i}$ is abelian.
\begin{lemma}\label{fivee}
Let $V_{i} < G_{i},$ suppose that $i\ge 4,$ and assume that
$G'/V_{i}$ is abelian. Then $D_{i}\leq E_{i}$.
\end{lemma}
\begin{proof}
We may  assume that $V_{i} =1.$ Hence, $ D_{i}= C_{G}(G_{i-1})$,
$G'$ is abelian, and $Y_{i}=Z(G).$ Since $ G'$ is abelian, we obtain
$[G, D_{i}, G_{i-2} ] \leq  [ G', G'] = 1$. On the other hand, we
have $[G_{i-2}, G, D_{i} ]=[ G_{i-1},D_{i}] =1$. By the Three
Subgroups Lemma, which is Lemma 8.27 in \cite{Isaacs2}, we get $
[D_{i}, G_{i-2}, G ]=1.$ Therefore, $[D_{i}, G_{i-2}] \leq Y_{i}$.
Now, we know that $[D_{i}, G_{i-2}]= [G_{i-2},D_{i}] \leq G_{i-1},$
and $[ D_{i},G_{i-2}] \leq G_{i-1}\cap Y_{i}$. We conclude that
$D_{i}\leq E_{i},$ as desired.
\end{proof}
In the next lemma, we get an upper bound for the index of $D_{i}$ in
$G.$
\begin{lemma}\label{seventeenprimee}
Assume that $V_{i} =1.$ If $|G_{i}|=p,$ then $|G:D_{i}|\le
|G_{i-1}:G_{i-1}\cap Y_{i}|$.
\end{lemma}
\begin{proof}
By Theorem 1 in \cite{Lewis1}, we know that $G_{i-1}/V_{i-1}$ is an
elementary abelian $p$-group. Hence, we can find $x_{1}, \cdots
,x_{t} \in G_{i-1}\setminus Y_{i}$, such that $G_{i-1} =\langle
x_{1}, \cdots ,x_{t}, G_{i-1}\cap Y_{i} \rangle$, where $ |G_{i-1} :
G_{i-1}\cap Y_{i}| =p^{t}$. Since $|G_{i}|=p$, we know by Lemma
\ref{twoone} that $|G:C_{G}(x_{j})|=p$ for all $j= 1, \cdots , t$.
Thus,
$$|G:D_{i}|=|G:\bigcap_{j=1}^{t} C_{G}(x_{j})|\le \prod_{j=1}^{t}|G:C_{G}(x_{j})|=p^{t} = |G_{i-1} : G_{i-1}\cap Y_{i}|.$$

\end{proof}

In our next lemma, we prove a very interesting isomorphism that will
a be a key to get the index of $V_{i}$ in $G_{i}.$
\begin{lemma}\label{seventeeen}
Assume that 
$ G_{i}$ is $H_{1}.$ Let $a \in G_{i-1} \setminus V_{i-1}$ and set
$K/V_{i}=C_{G/V_{i}}(aV_{i}).$ Then $ G/K \cong G_{i}/V_{i}$.
\end{lemma}
\begin{proof}
Without loss of generality, we may assume that $V_{i}=1$. Consider
the map from $G$ to $ G_{i}$ defined by $ g \mapsto [g, a]$. Since
$a \in G_{i-1}$, we have $[g, a] \in G_{i} $ for every $ g\in G.$
Hence, this map is well defined. Also, we know that $ G_{i}$ is
central in $G.$ Thus, this map is a homomorphism with kernel $K$. By
Lemma \ref{twotwo}, this map is onto. Therefore, by the First
Isomorphism Theorem, we conclude that $ G/K \cong G_{i}.$
\end{proof}

Now, we prove the following result.

\begin{cor}\label{seventeencomposite}
Assume that $ G_{i}$ is $H_{1}.$ Then $|G_{i}:V_{i}|\le |G:D_{i}|.$
\end{cor}
\begin{proof}
Let $a$ and $K$ be as in Lemma \ref{seventeeen}. We know since $a\in
G_{i-1}$ and $D_{i}/V_{i}=C_{G/V_{i}}(G_{i-1}/V_{i})$ that $D_{i}\le
K.$ Hence, $|G_{i}:V_{i}|=|G:K|\le |G:D_{i}|.$
\end{proof}

The following result is very useful to prove our main theorem.

\begin{lemma}\label{twentyy}
Assume that $V_{i}< G_{i}$, $G'/V_{i}$ is abelian, and $G_{i-1}$ is
$H_{1},$ for $i\ge 4.$ Let $a \in G_{i-2} \setminus V_{i-2}$ and set
$K/V_{i-1}=C_{G/V_{i-1}}(aV_{i-1}).$ Then $ K \leq D_{i}$.
\end{lemma}
\begin{proof}
We may assume that $V_{i}=1.$ Hence, $V_{i-1}$ is central in $G$,
$G'$ is abelian, $Y_{i}= Z(G),$ and $D_{i}=C_{G}(G_{i-1})$. Fix
$x\in K,$ and let $w\in G$ be arbitrary. Notice that $[a, x] \in
V_{i-1} \leq Y_{i}$. Thus, $[a , x, w]=1.$ Also, $[x,w]\in G'.$
Because $i\ge 4$, $G_{i-2}\le G'$ so $a\in G'$. Since $G'$ is
abelian, $[x, w, a ]\le [G',G']=1$. Therefore, by Hall's Identity,
which is Lemma 8.26 in \cite{Isaacs2}, we obtain $[w, a , x]=1$.
This implies that $x$ centralizes $[w,a].$ Since $a \not\in V_{i-2}$
and $G_{i-1}$ is $H_{1},$ we deduce by Lemma \ref{twotwo} that as
$w$ runs through all of $G$, $[w,a]$ runs through all of $G_{i-1}$.
Hence, $x$ centralizes $G_{i-1}.$ Thus, $ x\leq D_{i}.$ Therefore, $
K\leq D_{i}.$
\end{proof}

As a consequence of the previous lemma, we get the following
corollary.

\begin{cor}\label{corone}
Assume that $V_{i}< G_{i}$, $G'/V_{i}$ is abelian, and $G_{i-1}$ is
$H_{1},$ for $i\ge 4.$ Then $D_{i-1}\le D_{i}.$
\end{cor}
\begin{proof}
Let $a \in G_{i-2} \setminus V_{i-2}$ and set $
K/V_{i-1}=C_{G/V_{i-1}}(aV_{i-1}).$ Then by Lemma \ref{twentyy} we
have $ K \leq D_{i}.$ Also, we know that $D_{i-1}\le K.$ Thus,
$D_{i-1}\le D_{i}.$
\end{proof}
 We now get an upper bound for $|G_{i-1}:G_{i-1}\cap Y_{i}|.$
\begin{lemma}\label{fifty}
Assume that $V_{i} <G_{i}$ and $G_{i-1}$ is $H_{1}$. Then
$|G:E_{i}|\ge |G_{i-1}:G_{i-1}\cap Y_{i}|$.
\end{lemma}
\begin{proof}
Fix $a\in G_{i-2}\setminus V_{i-2},$ and consider the map $f$ from
$G$ to $G_{i-1}/V_{i-1}$ defined by $f(g)=[a,g]V_{i-1}.$ As in the
proof of Lemma \ref{seventeeen}, we know that $f$ is an onto
homomorphism. It follows that $f$ maps $G/E_{i}$ onto
$G_{i-1}/f(E_{i}).$ Thus, $|G_{i-1}:f(E_{i})| \leq |G:E_{i}|$. Since
$a\in G_{i-2},$ $[E_{i},a]\le G_{i-1} \cap Y_{i},$ and thus
$f(E_{i})\leq G_{i-1}\cap Y_{i}$. Then $|G_{i-1}:G_{i-1}\cap Y_{i}|
\leq |G_{i-1}:f(E_{i})|$. Hence, $|G:E_{i}|\ge |G_{i-1}:G_{i-1}\cap
Y_{i}|$ as required.
\end{proof}

\section{Proofs of Theorems 1, 2, and 3}

In this section, we prove our three theorems using the general
lemmas that we proved in the previous section.

Now, we  prove Theorem 1.

\begin{proof}[Proof of Theorem 1]
We have $D_{3}=D_{3}.$ This is the initial case of induction. Assume
that the theorem is true for $i-1.$ We are going to prove it for
$i.$ By hypothesis, we know that $G_{i}$ is $H_{1},$ and by Lemma
\ref{Hone}, we have $V_{i-1}=G_{i-1}\cap Y_{i}.$ This implies $E_{i}
=D_{i-1}.$ By the inductive hypothesis we know that $D_{i-1}=D_{3},$
and so, $E_{i}=D_{3}.$ By Lemma \ref{fivee}, we obtain $D_{i}\le
E_{i}.$ Applying Corollary \ref{corone}, we conclude that
$D_{i-1}\le D_{i}.$ Thus, $D_{i-1}\le D_{i} \le E_{i} =D_{i-1}.$
Therefore, we deduce that $D_{i}=E_{i}=D_{i-1}=D_{3}.$

\end{proof}

Now, we are ready to prove our second theorem.

\begin{proof}[Proof of Theorem 2]

We are going to prove this theorem by induction. Notice that the
initial case of induction $(i=3)$ is done by Lewis in \cite{Lewis1}.
Now, assume that the theorem is true for $k= i-1 $. We are going to
prove it for $k=i$. Also in this proof, without loss of generality,
we may assume that $V_{i}=1$. We also know by the inductive
hypothesis that $D_{i-1}=D_{3}$ and $G_{i-1}$ is $H_{1}.$ Now, by
Lemma \ref{fivee} we have that $D_{i} \le E_{i}.$ By Corollary
\ref{corone}, we have $D_{i}\le D_{i-1}$. First we assume that
$|G_{i}|=p.$ Thus, we obtain
$$|G:D_{i}|\ge |G:D_{i-1}| \ge  |G_{i-1}: V_{i-1}| \ge |G_{i-1}: G_{i-1}\cap Y_{i}|.$$
But by Lemma \ref{seventeenprimee}, we have $ |G:D_{i}|\le |G_{i-1}:
G_{i-1}\cap Y_{i}|.$ Hence, we have equality throughout the above
inequality. Therefore, $V_{i-1}=G_{i-1} \cap Y_{i},$ and
$|G_{i-1}:V_{i-1}|=|G:D_{i-1}|.$

Now, assume that $|G_{i}| >p.$ Consider a normal subgroup $N$, such
that $V_{i}\le N < G_{i}$ and $|G_{i}:N| =p$. The above argument
shows that $V_{i-1}(G/N)=Y_{i}(G/N)\cap (G_{i-1}/N).$ Thus, $G_{i}$
satisfies $H_{1}.$ By strong induction we have $G_{4},\cdots
,G_{i-1}$ satisfy $H_{1}.$ Thus, we may apply Theorem 1 to see that
$D_{i}=D_{3}.$ First define $D_{iN}/N=C_{G/N}(G_{i-1}/N).$ Note that
$D_{i}\le D_{iN},$ and so $D_{iN}=D_{3}.$ The above argument yields
$|G:D_{3}|=|G:D_{i-1}|=|G_{i-1}:V_{i-1}|.$ To prove part (d), since
$G_{i}$ is $H_{1},$ by Corollary \ref{seventeencomposite} we obtain
$|G_{i}:V_{i}|\le |G:D_{i}|,$ as desired.

\end{proof}

Now, we prove Theorem 3, which is a conclusive proof of MacDonald's
conjecture in \cite{MacDonald1} about the order of $G_{3},$ in the
case when $G$ is a Camina group of nilpotence class $3.$

\begin{proof}[Proof of Theorem 3] Note that $V_{1}=G_{2}.$ Hence,
$V_{2}=G_{3}.$ We deduce that $V_{3}=G_{4}=1.$ Let $a\in
G_{2}\setminus V_{2}$ and set $K =C_{G}(a).$ Thus, by Lemma
\ref{seventeeen}, we have $G_{3}\cong G/K.$ But $D_{3} \le K.$ By
MacDonald in \cite{MacDonald1}, we know that $|G:D_{3}| =p^{n}.$
Thus,
$$|G_{3}|= |G:K| \le |G:D_{3}| =p^{n}.$$
\end{proof}

\end{document}